\documentclass[12pt]{amsart}

\usepackage{amscd,amssymb,amsmath,graphicx,verbatim}
\usepackage{wasysym}
\usepackage{url}
\usepackage{doi}
\usepackage{hyperref}
\usepackage[TS1,OT1,T1]{fontenc}
\usepackage{lscape} 
\usepackage{fullpage} 
\usepackage{pslatex} 
\usepackage[all]{xy}

\newtheorem{theorem}{Theorem}[section]
\newtheorem{lemma}[theorem]{Lemma}

\newtheorem{proposition}[theorem]{Proposition}
\newtheorem{conjecture}[theorem]{Conjecture}

\theoremstyle{definition}
\newtheorem{definition}[theorem]{Definition}

\newtheorem{example}[theorem]{Example}

\theoremstyle{remark}
\newtheorem{remark}[theorem]{Remark}

\newcommand{\Dcal}{\ensuremath{\mathcal{D}}}

\newcommand{\Ccal}{\ensuremath{\mathcal{C}}}

\newcommand{\Acal}{\ensuremath{\mathcal{A}}}

\newcommand{\Ecal}{\ensuremath{\mathcal{E}}}

\newcommand{\Kbb}{\mathbb{K}}

\newcommand{\ra}{\rightarrow}

\numberwithin{equation}{section}

\begin{document}
\title{Homological embeddings for preprojective algebras}
\author{Frederik Marks}

\address{Frederik Marks, University of Stuttgart, Institute for Algebra and Number Theory, Pfaffenwaldring 57, 70569 Stuttgart, Germany}
\email{marks@mathematik.uni-stuttgart.de}

\thanks{The author is grateful to Martin Kalck for helpful discussions and for explaining some of the results in \cite{K}. Moreover, the author thanks Julian K\"ulshammer and Jorge Vit\'oria for useful comments on an earlier version of this text.}

\begin{abstract}
For a fixed finite dimensional algebra $A$, we study representation embeddings of the form $mod(B)\rightarrow mod(A)$. Such an embedding is called homological, if it induces an isomorphism on all Ext-groups and weakly homological, if only Ext$^1$ is preserved. In case $A$ is a preprojective algebra of Dynkin type, we give an explicit classification of all weakly homological and homological embeddings. Furthermore, we show that for self-injective algebras a classification of homological embeddings becomes accessible once these algebras fulfil the Tachikawa conjecture.
\end{abstract}
\maketitle


\section{Introduction}
One way to better understand the representation theory of a given finite dimensional algebra $A$ is to study embeddings of module categories $mod(B)\ra mod(A)$ where $B$ is another finite dimensional algebra. However, a classification of such embeddings is usually too much to ask for. In this article, we focus on studying fully faithful exact functors $mod(B)\ra mod(A)$ that yield isomorphisms either on only the first Ext-groups or, more generally, on all of them. These functors will be called (weakly) homological embeddings. Our main result provides a combinatorial classification of all weakly homological embeddings for preprojective algebras of Dynkin type. Furthermore, we describe explicitly which of these embeddings are homological.

Preprojective algebras of Dynkin type were first introduced and studied in \cite{GP} and they play an important role in modern representation theory, with applications to geometry, Lie theory and cluster algebras (compare, for example, \cite{BKT, CB, GLS, R}). Generally, preprojective algebras are of wild representation type. However, it was shown in \cite{AM} and \cite{Mi} that in the context of (generalised) tilting theory classification results can not only be obtained, but are, in fact, discrete. Here, an important role is played by the Weyl group associated to the graph defining the preprojective algebra. Following this philosophy, we prove the following result.\\

\noindent\textbf{Theorem A} (Theorem \ref{Main 1})

\noindent\textit{Let $A=A_Q$ be a preprojective algebra of Dynkin type. Then there is a 1-1 correspondence between
\begin{enumerate}
\item elements of the Weyl group $W_Q$;
\item equivalence classes of weakly homological embeddings $mod(B)\ra mod(A)$.
\end{enumerate}
Moreover, if $F_w:mod(B)\ra mod(A)$ is a weakly homological embedding corresponding to an element $w$ in $W_Q$, then the algebra $B$ is Morita equivalent to $A/I_v$ for some idempotent ideal $I_v$ associated to an element $v$ in $W_Q$ with $l(wv)=l(w)+l(v)$. In particular, $B$ is Morita equivalent to a direct product of preprojective algebras of Dynkin type.}\\

In order to better understand homological embeddings $mod(B)\ra mod(A)$ for a given finite dimensional and self-injective algebra $A$, it is helpful to know if $A$ fulfils the Tachikawa conjecture, i.e., if every finite dimensional $A$-module $X$ with $Ext_A^i(X,X)=0$ for all $i>0$ is already projective. In this case, a classification of homological embeddings becomes accessible and it is linked to the study of certain triangulated subcategories of the stable module category $\underline{mod}(A)$ (Proposition \ref{4.Prop2}). In this context, we show that blocks of group algebras do not admit any non-trivial homological embeddings (Theorem \ref{Main symmetric}). Finally, we complement Theorem A by obtaining a classification of the homological embeddings for preprojective algebras of Dynkin type.\\

\noindent\textbf{Theorem B} (Theorem \ref{4.Thm.2} and Example \ref{ex hom preproj})

\noindent\textit{Let $A$ be a preprojective algebra of Dynkin type and $F:mod(B)\ra mod(A)$ be a homological embedding that is neither zero nor an equivalence. Then $A$ is of type $\Acal_n$ $({\small{\xymatrix{1\ar@{-}[r] & 2\ar@{-}[r] & \dots\ar@{-}[r] & n}}})$ for $n\ge 2$ and the algebra $B$ is Morita equivalent to $\Kbb$\,.
In fact, for each $n\ge 2$ there are precisely two such choices for $F$, up to equivalence, which correspond to the Weyl group elements $s_{n-1}(s_{n-2}s_{n-1})\cdots (s_1s_2\cdots s_{n-1})$ and $s_n(s_{n-1}s_n)\cdots (s_2s_3\cdots s_n)$, respectively.}\\

Such classification results are relevant to better understand the structure of the derived category of all $A$-modules (see Remark \ref{rem}). In fact, every homological embedding $mod(B)\ra mod(A)$ gives rise to a fully faithful functor between the associated derived categories and, moreover, to a recollement of triangulated categories (see \cite{GL} and \cite{P}).\\

The structure of the paper is as follows. In Section 2, we fix some necessary notation. Section 3 is dedicated to the notion of (weakly) homological embedding and discusses possible ways of classifying them (Proposition \ref{Prop ring epi}). In Section 4, we introduce preprojective algebras of Dynkin type and prove Theorem A. Section 5 deals with the general situation of studying homological embeddings for self-injective algebras and how this problem is related to the Tachikawa conjecture. Finally, Section 6 contains the proof of Theorem B.


\section{Notation}
Throughout, $A$ will be a finite dimensional algebra over an algebraically closed field $\Kbb$. The category of all finite dimensional left $A$-modules is denoted by $mod(A)$. All subcategories $\Ccal$ of $mod(A)$ are considered to be full. A subcategory $\Ccal$ of $mod(A)$ is said to be \textbf{functorially finite}, if every $A$-module admits a left and a right $\Ccal$-approximation, i.e., for every $X\in mod(A)$ there are objects $C_1,C_2\in\Ccal$ and morphisms $g_1:X\ra C_1$ and $g_2:C_2\ra X$ such that the maps $Hom_A(g_1,\tilde{C})$ and $Hom_A(\tilde{C},g_2)$ are surjective for all $\tilde{C}\in\Ccal$. A subcategory $\Ccal$ of $mod(A)$ is called \textbf{wide}, if it is closed for kernels, cokernels and extensions in $mod(A)$ and it is said to be a \textbf{torsion class}, if it is closed for quotients and extensions. 
Moreover, for an $A$-module $X$, we denote by $add(X)$ the subcategory of $mod(A)$ consisting of all direct summands of finite direct sums of copies of $X$ and by $gen(X)$ (respectively, $sub(X)$) the subcategory of $mod(A)$ containing all quotients (respectively, all submodules) of finite direct sums of copies of $X$.

Let $\underline{mod}(A)$ be the \textbf{stable module category} defined to be the quotient of $mod(A)$ by the ideal generated by all the homomorphisms that factor through a projective module. By $\pi_A$ we denote the canonical quotient functor $mod(A)\ra \underline{mod}(A)$. It is well-known that in case $A$ is self-injective, $\underline{mod}(A)$ carries the structure of a triangulated category with shift functor $\Omega^{-1}$ (compare \cite{H}), where
$\Omega$ denotes the usual syzygy. Also recall that the Auslander-Reiten translate $\tau$, the Nakayama functor $\nu$ and the syzygy functor $\Omega$ induce autoequivalences on 
$\underline{mod}(A)$ related by the natural isomorphisms of functors $\tau\cong\Omega^2\nu\cong\nu\Omega^2$ and $\tau^{-1}\cong\Omega^{-2}\nu^{-1}\cong\nu^{-1}\Omega^{-2}$ (\cite[IV., Theorem 8.5]{SY}).


\section{Homological embeddings}

The definition of homological embedding for abelian categories is due to \cite{P}.

\begin{definition}
Let $A$ and $B$ be finite dimensional $\Kbb$-algebras. A fully faithful and exact functor $F:mod(B)\ra mod(A)$ is called a \textbf{(weakly) homological embedding}, if $F$ induces an isomorphism 
$$Ext_B^i(X,Y)\cong Ext_A^i(F(X),F(Y))$$ 
for all $i>0$ (respectively, for $i=1$) and for all $B$-modules $X$ and $Y$. Two (weakly) homological embeddings $F:mod(B)\ra mod(A)$ and $F':mod(B')\ra mod(A)$ are said to be \textbf{equivalent}, if there is  an equivalence of categories $G:mod(B)\ra mod(B')$ making the following diagram commute
$$\xymatrix{mod(B)\ar[rr]^F\ar[dr]_{G} & & mod(A) \\ & mod(B')\ar[ur]_{F'} & }$$
In particular, $B$ is Morita equivalent to $B'$ and the essential image of $F$, denoted by $Im(F)$, coincides with the essential image of $F'$.
\end{definition}

In this article, we are interested in classifying, up to equivalence, all possible (weakly) homological embeddings $mod(B)\ra mod(A)$ for a fixed algebra $A$. Note that without the assumption of having at least an isomorphism on the first Ext-groups this is an almost hopeless task. For example, if $A$ is a strictly wild algebra, then for every finite dimensional algebra $B$ there is a fully faithful exact functor $mod(B)\ra mod(A)$. However, the additional assumption of the functor being (weakly) homological will make classification results accessible.

Recall that a ring homomorphism $A\ra B$ between two finite dimensional $\Kbb$-algebras $A$ and $B$ is called a \textbf{ring epimorphism}, if it is an epimorphism in the category of rings. Moreover, two ring epimorphisms $f:A\ra B$ and $f':A\ra B'$ are said to be \textbf{equivalent}, if there is a (unique) isomorphism of algebras $g:B'\ra B$ making the following diagram commute
$$\xymatrix{A\ar[rr]^{f}\ar[dr]_{f'} & & B\\ & B'\ar[ur]_g &}$$

The next result is helpful to classify equivalence classes of (weakly) homological embeddings.

\begin{proposition}\label{Prop ring epi}
Let $A$ be a finite dimensional $\Kbb$-algebra. There are 1-1 correspondences between
\begin{enumerate}
\item equivalence classes of weakly homological embeddings $mod(B)\ra mod(A)$;
\item functorially finite wide subcategories of $mod(A)$;
\item equivalence classes of ring epimorphisms $A\ra B$ with $Tor_1^A(B,B)=0$.
\end{enumerate}

Moreover, the above correspondence $(1)\leftrightarrow (3)$ restricts to a bijection between
\begin{enumerate}
\item[(1')] equivalence classes of homological embeddings $mod(B)\ra mod(A)$;
\item[(3')] equivalence classes of ring epimorphisms $A\ra B$ with $Tor_i^A(B,B)=0$ for all $i>0$.
\end{enumerate}

\end{proposition}

\begin{proof}
The assignments for the first three bijections are as follows:
$$\begin{array}{ccccl}
\text{Bijection} &&&& \text{Assignment}\\
\hline
 (1)\rightarrow (2) &&&&  F\,\mapsto\; Im(F)\\
 (2)\rightarrow (3) &&&&  W\,\mapsto\; (A\ra End_A^{op}(C)) \text{ with } A\ra C \text{ minimal left W-approximation }\\
 (3)\rightarrow (1) &&&&  (A\ra B)\,\mapsto\; \text{restriction functor } mod(B)\ra mod(A) 
\end{array}$$
Clearly, the essential image of a weakly homological embedding yields a functorially finite wide subcategory of $mod(A)$ and this assignment is injective. The bijection $(2)\ra (3)$ is given by \cite[Theorem 1.6.1(2)]{I} using that $W$ being extension-closed translates to the vanishing condition on $Tor_1^A$ (compare \cite[Theorem 4.8]{Sch}). The inverse $(3)\ra (2)$ factors through $(1)$ by mapping a ring epimorphism $A\ra B$ to the essential image of its restriction functor $F:mod(B)\ra mod(A)$, where $F$ is well-known to be fully faithful and exact. In particular, the assignment $(1)\ra (2)$ is surjective.
Finally, the correspondence $(1')\leftrightarrow (3')$ follows from above, now using \cite[Theorem 4.4]{GL}.
\end{proof}

This proposition shows that equivalence classes of (weakly) homological embeddings always have a \textit{good} representative, namely the restriction functor of an associated ring epimorphism. The following remark collects some easy observations on (weakly) homological embeddings.

\begin{remark}\label{rk1}\hfill
\begin{enumerate}
\item If $A$ is hereditary, then every weakly homological embedding $mod(B)\ra mod(A)$ is already homological. In particular, the algebra $B$ is hereditary.
\item If $A$ is local, then every weakly homological embedding $F:mod(B)\ra mod(A)$ is trivial, i.e., $F$ is either zero or an equivalence. This follows from the fact that $F\not= 0$ already implies, using top-to-socle factorisation, that the unique simple $A$-module $S$ lies in $Im(F)$.
\item If $F:mod(B)\ra mod(A)$ is a weakly homological embedding, the following are equivalent.
\begin{enumerate}
\item $Im(F)$ is closed for subobjects;
\item $Im(F)$ is closed for quotients;
\item $F$ is equivalent to the restriction functor $mod(A/I)\ra mod(A)$ for some $I=I^2\lhd A$.
\end{enumerate}
The only non-trivial implication is $(a),(b)\Rightarrow (c)$ which follows from Proposition \ref{Prop ring epi} using \cite[Lemma 3.1]{M}.
\end{enumerate}
\end{remark}


\section{Weakly homological embeddings for preprojective algebras}
In this section, we fix $A$ to be a preprojective algebra of Dynkin type. Let us recall the definition. First, take a Dynkin quiver $Q=(Q_0,Q_1)$ of type $\Acal_n (n\ge 1)$, $\Dcal_n (n\ge 4)$, $\Ecal_6, \Ecal_7$ or $\Ecal_8$ and denote by $\overline{Q}$ its double quiver obtained by adding an arrow $\alpha^*:j\ra i$ for each arrow $\alpha:i\ra j$ in $Q_1$. The \textbf{preprojective algebra} $A=A_Q$ associated to $Q$ is given by the quotient of the path algebra $\Kbb \overline{Q}$ by the ideal $I$ generated by
$$\underset{\alpha\in Q_1}{\sum}(\alpha\alpha^*-\alpha^*\alpha).$$ Since the quiver $Q$ is Dynkin, the algebra $A$ is finite dimensional and self-injective. Moreover, the stable module category $\underline{mod}(A)$ has Calabi-Yau dimension 2, i.e., there is a natural isomorphism of functors $\nu\cong\Omega^{-3}$ (see \cite{ES}). In particular, it follows that $\tau\cong\Omega^{-1}$ and that we have a natural isomorphism $Ext_A^1(M,N)\cong DExt^1_A(N,M)$ for all $A$-modules $M,N$, where $D:=Hom_{\Kbb}(-,\Kbb)$.\\

Now let $Q$ be a Dynkin quiver with $Q_0=\{1,...,n\}$. The Weyl group $W_Q$ associated to $Q$ is defined by the generators $s_1,...,s_n$ and the relations 
\begin{itemize}
\item $s_i^2=1$;
\item $s_is_j=s_js_i$, if there is no edge connecting $i$ and $j$;
\item $s_is_js_i=s_js_is_j$, if there is an edge connecting $i$ and $j$.
\end{itemize}
We say that an expression $w=s_{i_1}\cdot\cdot\cdot s_{i_p}$ is \textbf{reduced}, if $p$ is minimal among all such presentations of $w$ in $W_Q$. We call $l(w):=p$ the \textbf{length} of $w$. Moreover, to every generator $s_i$ of the Weyl group we associate an ideal $I_{i}:=A(1-e_i)A$ of the preprojective algebra $A=A_Q$, where $e_i$ denotes the primitive idempotent corresponding to vertex $i$. The following result will be crucial in our context.

\begin{theorem}[{\cite[Theorem 2.7]{AIR}}, {\cite[Theorem 2.21]{Mi}}]\label{Thm Mizuno}
Let $A=A_Q$ be a preprojective algebra of Dynkin type. Then there is a 1-1 correspondence between
\begin{enumerate}
\item elements of the Weyl group $W_Q$;
\item torsion classes in $mod(A)$
\end{enumerate}
by mapping a reduced expression $w=s_{i_1}\cdot\cdot\cdot s_{i_p}$ to the torsion class $gen(I_w:=I_{i_1}\cdot\cdot\cdot I_{i_p})$.
\end{theorem}

Now we are able to state the main result of this section.

\begin{theorem}\label{Main 1}
Let $A=A_Q$ be a preprojective algebra of Dynkin type. Then there is a 1-1 correspondence between
\begin{enumerate}
\item elements of the Weyl group $W_Q$;
\item equivalence classes of weakly homological embeddings $mod(B)\ra mod(A)$.
\end{enumerate}
Moreover, if $F_w:mod(B)\ra mod(A)$ is a weakly homological embedding corresponding to an element $w$ in $W_Q$, then the algebra $B$ is Morita equivalent to $A/I_v$ for some element $v$ in $W_Q$ with $l(wv)=l(w)+l(v)$ and such that the ideal $I_v$ is idempotent. In particular, $B$ is Morita equivalent to a direct product of preprojective algebras of Dynkin type.
\end{theorem}

\begin{proof}
Let $w$ be an element of $W_Q$ with corresponding torsion class $gen(I_w)$. We consider the $gen(I_w)$-approximation sequence
$$\xymatrix{A\ar[r]^\phi & T_0\ar[r] & T_1\ar[r] & 0}$$ where the map $\phi$ is chosen minimal, i.e., every map $h\in End_A(T_0)$ fulfilling $h\circ\phi=\phi$ is already an isomorphism.
Using Theorem \ref{Thm Mizuno} and \cite[Corollary 3.11]{MS}, it follows that the assignment $w\mapsto gen(I_w)\cap T_1^\circ$ yields a bijection between elements of $W_Q$ and functorially finite wide subcategories of $mod(A)$, where $T_1^\circ:=\{X\in mod(A)\mid Hom_A(T_1,X)=0\}$. Therefore, we can use Proposition \ref{Prop ring epi} to show the wanted correspondence between elements of the Weyl group and equivalence classes of weakly homological embeddings $mod(B)\ra mod(A)$. 

Next, we consider the algebra $B$ occurring in $F_w: mod(B)\ra mod(A)$. Note that, associated to $w\in W_Q$, we obtain two torsion pairs, namely $(gen(I_w), I_w^\circ)$ and $(gen(T_1), T_1^\circ)$. Using Theorem \ref{Thm Mizuno}, there is some $u\in W_Q$ such that $gen(T_1)=gen(I_u)$. Consequently, keeping in mind the above assignments, the functor $F_w$ induces an equivalence
$$mod(B)\cong gen(I_w)\cap I_u^\circ.$$
Since the torsion class $gen(I_u)$ embeds into $gen(I_w)$, there is some $v\in W_Q$ such that $u=wv$ and $l(u)=l(w)+l(v)$. A dual version of this statement can be found in \cite[Section 5.2]{L}. In fact, observe that the torsion class $C_w$ in \cite{L} coincides with the $\Kbb$-dual of $I_{w^{-1}}^\circ$ in our notation (see also \cite[Example 5.6.(iii)-(iv)]{BKT}). Therefore, by \cite[Proposition 5.16 and Example 5.13]{BKT}, we obtain an equivalence of categories
$$\xymatrix{gen(I_w)\cap I_u^\circ\ar[rr]^{\!\!\!Hom_A(I_w,-)} & & I_v^\circ=sub(A/I_v).}$$
Moreover, since the module $I_w$ is Ext-projective in $gen(I_w)$ (compare \cite[Theorem 2.2]{Mi}), the above functor $Hom_A(I_w,-)$ is exact.
It follows that composition of $F_w$ and $Hom_A(I_w,-)$ yields an equivalence $mod(B)\cong sub(A/I_v)$ that preserves exact sequences. Consequently, $sub(A/I_v)$ is also closed for quotients in $mod(A)$ and, thus, it is equivalent to $mod(A/I_v)$. This proves that the algebra $B$ is Morita equivalent to $A/I_v$ and that both algebras are necessarily self-injective.

It remains to show that the ideal $I_v$ is idempotent and, therefore, $B$ is Morita equivalent to a direct product of preprojective algebras of Dynkin type. Since $sub(A/I_v)$ is a torsion-free class in $mod(A)$, it is, in particular, closed for extensions and, therefore (see above), it is a functorially finite and wide subcategory of $mod(A)$. Using Proposition \ref{Prop ring epi}, it follows that $$0=Tor_1^A(A/I_v,A/I_v)=I_v/I_v^2.$$ Thus, there is some idempotent $e\in A$ such that $I_v=AeA$. Now the claim follows by observing that every idempotent quotient of $A$ (obtained by deleting certain vertices in the quiver of the underlying bound path algebra) is again a direct product of preprojective algebras.
\end{proof}

\begin{remark}
Note that, for a preprojective algebra $A=A_Q$ of Dynkin type and an element $v$ in $W_Q$, in general, the ideal $I_v$ is not idempotent and the quotient $A/I_v$ is not isomorphic to a direct product of preprojective algebras. In fact, the arguments used in the proof of Theorem \ref{Main 1} show that the ideal $I_v$ is idempotent if and only if the algebra $A/I_v$ is self-injective.
\end{remark}

We illustrate Theorem \ref{Main 1} in the following example.

\begin{example}\label{ex A3}
Let $A$ be the preprojective algebra of type $\Acal_3$ $({\small{\xymatrix{1\ar@{-}[r] & 2\ar@{-}[r] & 3}}})$. Note that $A$ is representation-finite and its Auslander-Reiten quiver is given by
$$
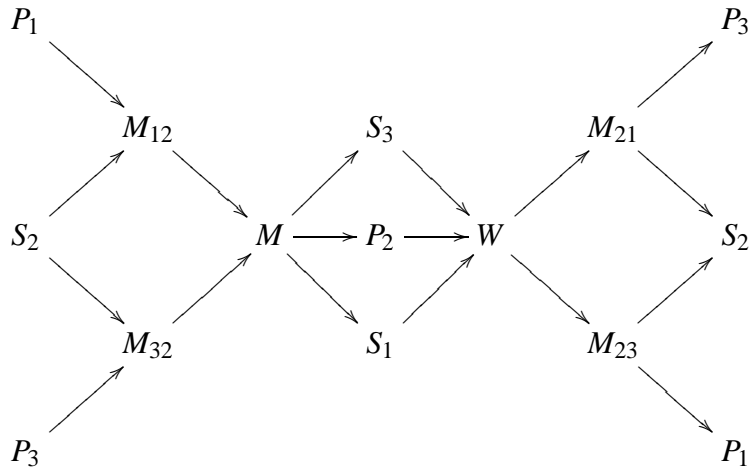
\xymatrix{P_1\ar[dr] & & & & & & P_3\\ & M_{12}\ar[dr] & & S_3\ar[dr] & & M_{21}\ar[ur]\ar[dr] & \\ S_2\ar[ur]\ar[dr] & & M\ar[ur]\ar[r]\ar[dr] & P_2\ar[r] & W\ar[ur]\ar[dr] & & S_2 \\ & M_{32}\ar[ur] & & S_1\ar[ur] & & M_{23}\ar[ur]\ar[dr] & \\ P_3\ar[ur] & & & & & & P_1}$$
The Weyl group $W_{\Acal_3}$ is known to be isomorphic to the symmetric group $\Sigma_4$. An isomorphism is, for example, given by the assignment $s_1\mapsto (34)$, $s_2\mapsto (23)$ and $s_3\mapsto (12)$. The following table lists all elements of the Weyl group, their associated torsion classes in $mod(A)$ and their corresponding equivalence classes of weakly homological embeddings $F:mod(B)\ra mod(A)$ (indicated by the essential image of $F$). Finally, it includes a representative of the Morita class of $B$. 

\vspace*{0.2cm}\center
\begin{tabular}{|c|c|c|c|}
\textbf{elements of $\Sigma_4$} & \textbf{torsion classes} & \textbf{wide subcategories} & \textbf{Morita class} \\ \hline $1_{\Sigma_4}$ & $mod(A) $ & $mod(A)$ & type $\Acal_3$ \\ $(12)$ & $gen(P_1\oplus P_2\oplus M_{21})$ & $add(P_1\oplus S_1\oplus W\oplus M_{23})$ & type $\Acal_2$ \\ $(13)$ & $gen(P_1\oplus M_{12}\oplus S_1)$ & $add(P_1)$ & $\Kbb$ \\ $(14)$ & $add(S_2)$ & $add(S_2)$ & $\Kbb$ \\ $(23)$ & $gen(P_1\oplus P_3\oplus M)$ & $add(M_{12}\oplus M_{32})$ & $\Kbb\times\Kbb$ \\ $(24)$ & $gen(P_3\oplus M_{32}\oplus S_3)$ & $add(P_3)$ & $\Kbb$ \\ $(34)$  & $gen(P_3\oplus P_2\oplus M_{23})$ & $add(P_3\oplus S_3\oplus M_{21}\oplus W)$ & type $\Acal_2$ \\ $(12)(34)$ & $gen(P_2\oplus M_{21}\oplus M_{23})$ & $add(W)$ & $\Kbb$ \\ $(13)(24)$ & $add(S_1\oplus S_3)$ & $add(S_1\oplus S_3)$ & $\Kbb\times\Kbb$ \\ $(14)(23)$ & $\{0\}$ & $\{0\}$ & $0$ \\ $(123)$ & $gen(P_1\oplus M_{21}\oplus M_{12})$ & $add(P_1\oplus S_2)$ & $\Kbb\times\Kbb$ \\ $(124)$ & $gen(M_{21}\oplus S_2)$ & $add(M_{21})$ & $\Kbb$ \\ $(132)$ & $gen(P_1\oplus M\oplus S_1)$ & $add(P_1\oplus S_3\oplus M_{12}\oplus M)$ & type $\Acal_2$ \\ $(134)$ & $gen(M_{12}\oplus S_1)$ & $add(M_{12})$ & $\Kbb$ \\ $(142)$ & $gen(M_{32}\oplus S_3)$ & $add(M_{32})$ & $\Kbb$ \\ $(143)$ & $gen(M_{23}\oplus S_2)$ & $add(M_{23})$ & $\Kbb$ \\ $(234)$ & $gen(P_3\oplus S_3\oplus M)$ & $add(P_3\oplus S_1\oplus M_{32}\oplus M)$ & type $\Acal_2$ \\ $(243)$ & $gen(P_3\oplus M_{23}\oplus M_{32})$ & $add(P_3\oplus S_2)$ & $\Kbb\times\Kbb$ \\ $(1234)$ & $gen(M_{12}\oplus M_{21})$ & $add(S_1\oplus S_2\oplus M_{12}\oplus M_{21})$ & type $\Acal_2$ \\ $(1243)$ & $gen(M_{21}\oplus M_{23}\oplus S_2)$ & $add(M_{21}\oplus M_{23})$ & $\Kbb\times\Kbb$ \\ $(1324)$ & $add(S_1)$ & $add(S_1)$ & $\Kbb$ \\ $(1342) $ & $gen(M\oplus S_1\oplus S_3)$ & $add(M)$ & $\Kbb$ \\ $(1423)$  & $add(S_3)$ & $add(S_3)$ & $\Kbb$  \\ $(1432)$ & $gen(M_{32}\oplus M_{23})$ & $add(S_2\oplus S_3\oplus M_{32}\oplus M_{23})$ & type $\Acal_2$  
\end{tabular}
\end{example}
\vspace{0.2cm}

\begin{remark}
The bijection established in Theorem \ref{Main 1} can be put into a larger context by comparing it to further classification results in which the Weyl group $W_Q$ ($Q$ Dynkin) plays a prominent role (see \cite[Theorem 4.1]{Mi}).
It was shown in \cite{ORT} that there is a one-to-one correspondence between the elements of $W_Q$ and the additive quotient-closed subcategories of $mod(\Kbb Q)$. The bijection is given by mapping an element $w$ in $W_Q$ to the full subcategory $add(\Theta(I_w))$ of $mod(\Kbb Q)$ where 
$$\Theta:mod(A_Q)\ra mod(\Kbb Q)$$ 
denotes the restriction functor induced by the natural inclusion $\Kbb Q\ra A_Q$. Recall that $\Theta$ is always faithful and dense. It is, moreover, worth mentioning that the functor $\Theta$ is less useful when directly applied to a weakly homological embedding $F:mod(B)\ra mod(A_Q)$. It turns out that the full subcategory $add(\Theta(F(B)))$ of $mod(\Kbb Q)$ is, in general, neither closed for kernels or cokernels nor for extensions. Moreover, there can be non-equivalent weakly homological embeddings $F$ and $F'$ with $add(\Theta(F(B)))=add(\Theta(F'(B')))$ (compare Example \ref{ex A3}). 

Nevertheless, the functor $\Theta$ can be used to construct examples of weakly homological embeddings. Since $\Theta$ is exact, the preimage of every wide subcategory in $mod(\Kbb Q)$ under $\Theta$ is also wide and, by \cite[Corollary 3.11]{MS}, functorially finite in $mod(A_Q)$. Hence, there is an induced weakly homological embedding following Proposition \ref{Prop ring epi}. However, it can be checked that, in general, not all weakly homological embeddings $mod(B)\ra mod(A_Q)$ arise in this way (even if we run through all possible orientations of the quiver $Q$).
\end{remark}


\section{Homological embeddings and the Tachikawa conjecture}
This section is an intermediate step towards the classification of homological embeddings for preprojective algebras. In fact, the ideas discussed here will be essential to prove Theorem \ref{4.Thm.2}.
Throughout, $A$ will be a self-injective algebra and we study the relationship between homological embeddings $mod(B)\ra mod(A)$ and the Tachikawa conjecture (see \cite[Section 8]{T}).

\begin{conjecture}[Tachikawa]
Let $A$ be a self-injective algebra and let $M$ be a finite dimensional $A$-module. If $Ext_A^i(M,M)=0$ for all $i>0$, then $M$ is projective.
\end{conjecture}

There are several affirmative answers to the conjecture. Namely, it is known to hold for
\begin{itemize}
\item group algebras of finite groups (\cite[Chapter 3]{S});
\item self-injective algebras of finite representation type (\cite[Chapter 3]{S});
\item symmetric algebras with radical cube zero (\cite[Theorem 3.1]{Ho});
\item local self-injective algebras with radical cube zero (\cite[Theorem 3.4]{Ho}).
\end{itemize}

Recall that an algebra $A$ is called \textbf{periodic}, if $A$ is a periodic $A\mbox{-}A$-bimodule with respect to the syzygy $\Omega$. In this case, $A$ is a self-injective algebra and also every left $A$-module without projective direct summands is periodic with respect to $\Omega$. 

\begin{lemma}\label{4.Lem2}
Let $A$ be a self-injective algebra.
\begin{enumerate}
\item[(1)] If $A$ is periodic, then $A$ fulfils the Tachikawa conjecture.
\item[(2)] Let $B$ be a self-injective algebra such that $\underline{mod}(B)$ is triangle equivalent to $\underline{mod}(A)$. Then $B$ fulfils the Tachikawa conjecture if and only if so does $A$.
\end{enumerate}
\end{lemma}

\begin{proof}
We first show (1). Let $M$ be an indecomposable non-projective module in $mod(A)$ and take $d>0$ such that $\Omega^d(M)=M$. It suffices to check that $Ext_A^d(M,M)\not= 0$. But this follows from \cite[IV., Theorem 9.6]{SY} using the $\Kbb$-linear isomorphism
$$0\not=\underline{Hom}_A(M,M)=\underline{Hom}_A(\Omega^d(M),M)\cong Ext_A^d(M,M).$$

We now prove (2). Denote by $\psi:\underline{mod}(B)\ra \underline{mod}(A)$ the triangle equivalence between the two stable module categories. Since statement (2) is symmetric, we only prove one implication. Assume that the algebra $A$ fulfils the Tachikawa conjecture. Let $N$ be an indecomposable non-projective $B$-module. By assumption, we know that there is some $d>0$ such that
$$Ext_A^d(\psi(N),\psi(N))\not= 0.$$
Using the same $\mathbb{K}$-linear isomorphism as in (1), it follows that
$$0\not=Ext_A^d(\psi(N),\psi(N))\cong\underline{Hom}_A(\Omega_A^d(\psi(N)),\psi(N))\cong\underline{Hom}_B(\Omega_B^d(N),N)\cong Ext_B^d(N,N).$$
Consequently, also $B$ fulfils the Tachikawa conjecture.
\end{proof}

Next, we relate the Tachikawa conjecture to the study of homological embeddings.

\begin{proposition}\label{4.Prop2}
Let $A$ be a self-injective algebra that fulfils the Tachikawa conjecture and let $F:mod(B)\ra mod(A)$ be a fully faithful exact functor. Then the following are equivalent.
\begin{enumerate}
\item[(1)] $F$ is homological;
\item[(2)] $F(B)$ is a projective $A$-module;
\item[(3)] $Im(F)$ is closed under syzygies in $mod(A)$ ($X\in Im(F)\Rightarrow\Omega_A^s(X)\in Im(F)$ for all $s\in\mathbb{Z}$);
\item[(4)] $B$ is a self-injective algebra and $F$ induces a fully faithful triangle functor $F_{\Delta}$ making the following diagram commute 
$$\xymatrix{mod(B)\ar[r]^{F}\ar[d]_{\pi_B} & mod(A)\ar[d]^{\pi_A}\\\underline{mod}(B)\ar[r]^{F_{\Delta}} & \underline{mod}(A)}$$
In particular, $\pi_A(Im(F))$ is a triangulated subcategory of $\underline{mod}(A)$.
\end{enumerate}
Moreover, if these conditions are satisfied, then also $B$ fulfils the Tachikawa conjecture. 
\end{proposition}

\begin{proof}
First assume that (1) holds. Since $F$ is homological, we know that 
$$Ext_A^i(F(B),F(B))\cong Ext_B^i(B,B)=0$$ 
for all $i>0$. Consequently, by assumption on $A$, the module $F(B)$ is projective. 
Next, note that if $F(B)$ is a projective $A$-module, it is also an injective $A$-module. In particular, since $F$ is fully faithful and exact, $B$ is an injective $B$-module and, hence, the algebra $B$ is self-injective.
Moreover, it follows that condition (2) is equivalent to stating that the category $Im(F)$ is closed under taking projective covers and injective envelopes in $mod(A)$. Since $Im(F)$ is also closed for kernels, cokernels and extensions in $mod(A)$, (2) is equivalent to (3). In particular, computing Ext-groups in $Im(F)$ coincides with computing Ext-groups in $mod(A)$ and, thus, we obtain back (1). Furthermore, observe that the diagram of categories in (4) can only commute, if $F(B)$ is projective. Hence, (4) implies (2). Now assume that (1)-(3) hold. We have to show (4). We define the functor $F_{\Delta}$ to be $F$ on objects and to map a morphism $\pi_B(g)\in\underline{mod}(B)$ to $\pi_A(F(g))\in \underline{mod}(A)$. First of all, $F_\Delta$ is a well-defined functor, since, by (2), every morphism of $B$-modules that factors through a projective $B$-module is mapped to a morphism of $A$-modules that factors through a projective $A$-module in $add(F(B))$. Moreover, $F_\Delta$ is full, since so is $F$. Next, we check that $F_\Delta$ is faithful. Let $F(g):F(X)\ra F(Y)$ be a morphism in $mod(A)$ that factors through a projective $A$-module $P$
$$\xymatrix{F(X)\ar[rr]^{F(g)}\ar[dr] & & F(Y)\\ & P.\ar[ur]_h &}$$
Then the map $h$ factors through the minimal left $Im(F)$-approximation $P\ra F(C_P)$ for some $C_P$ in $mod(B)$. Note that the minimality of the approximation implies that $C_P$ is a projective $B$-module. Consequently, the map $g$ is zero in $\underline{mod}(B)$. Finally, $F_\Delta$ is a triangle functor, since, by (2) and (3), $F$ induces a natural isomorphism $\Omega_B\cong\Omega_A$.

It remains to check that also $B$ fulfils the Tachikawa conjecture. Consider a $B$-module $X$ with $Ext_B^i(X,X)=0$ for all $i>0$. Since $F$ is homological and $A$ fulfils the Tachikawa conjecture, $F(X)$ is a projective $A$-module and, thus, $X$ must be a projective $B$-module.
\end{proof}

In general, it seems that self-injective algebras do not admit many homological embeddings. The next lemma points in this direction. Recall that a two-sided idempotent ideal $I$ of $A$ is called \textbf{stratifying}, if restriction induces a homological embedding $mod(A/I)\ra mod(A)$.

\begin{lemma}
Let $A$ be a connected self-injective algebra fulfilling the Tachikawa conjecture and let $F:mod(B)\ra mod(A)$ be a homological embedding such that $Im(F)$ is closed for quotients. Then $F$ is either zero or an equivalence.
In particular, all stratifying ideals of $A$ are trivial.
\end{lemma}

\begin{proof}
By Remark \ref{rk1}(3), the functor $F$ is equivalent to the restriction functor $mod(A/I)\ra mod(A)$ for some stratifying ideal $I\lhd A$. Let us assume that $F$ is not zero, i.e., $I\not= A$. By Proposition \ref{4.Prop2}, the $A$-module $A/I$ is projective, and we get a decomposition of the regular module $A$ into $I\oplus A/I$. Since the algebra $A$ is self-injective (every arrow of the underlying quiver is part of an oriented cycle) and $Hom_A(I,A/I)=0$, it follows that also $Hom_A(A/I,I)=0$. Consequently, $A$ splits as an algebra into the direct product of $End_A(I)$ and $A/I$. By assumption, this implies that $I$ is zero and, thus, $F$ is an equivalence.
\end{proof}

\begin{remark}\label{rem}
Note that the lack of stratifying ideals for an algebra $A$ or, more generally, the absence of homological embeddings $mod(B)\ra mod(A)$ is related to the notion of derived simplicity. Recall that $A$ is said to be \textbf{derived simple}, if its derived category does not appear as the middle term of a non-trivial recollement of derived module categories (see \cite{AKLY} for details). It is well-known that every stratifying ideal $I=AeA$ of $A$ with $e=e^2\in A$ induces a recollement
$$\xymatrix{\Dcal(A/AeA)\ar[r]^{}& \Dcal(A)\ar@<1.5ex>[l]_{}\ar@<-1.5ex>[l]_{}\ar[r]^{}&
\Dcal(eAe)\ar@<1.5ex>_{}[l]\ar@<-1.5ex>_{}[l]}$$
where $\Dcal(A)$ denotes the derived category of all (not necessarily finite dimensional) left $A$-modules.
In particular, an algebra $A$ with a non-trivial stratifying ideal is not derived simple. More generally, most examples of recollements of derived module categories having $\Dcal(A)$ as middle term do arise from homological embeddings $F:mod(B)\ra mod(A)$ by first extending $F$ to a functor between the category of all $B$-modules and the category of all $A$-modules and then deriving it. 
\end{remark}

In \cite{LY}, it was shown that blocks of group algebras are derived simple. Motivated by this result, we provide a classification of the homological embeddings for such algebras.
Recall that an algebra $A$ is called \textbf{symmetric}, if $A$ is isomorphic to $Hom_{\Kbb}(A,\Kbb)$ as an $A\mbox{-}A$-bimodule and $A$ is called \textbf{weakly symmetric}, if $top(P)\cong soc(P)$ for all indecomposable projective modules $P$.
Note that symmetric algebras are weakly symmetric and weakly symmetric algebras are self-injective.

\begin{theorem}\label{Main symmetric}
Let $A$ be a connected weakly symmetric algebra fulfilling the Tachikawa conjecture (e.g., take $A$ to be the block of a group algebra) and let $F:mod(B)\ra mod(A)$ be a homological embedding. Then $F$ is either zero or an equivalence.
\end{theorem}

\begin{proof}
Without loss of generality, we can assume that $A$ is basic and, thus, given as a bound path algebra $\Kbb Q/I$. If we assume $F$ to be non-zero, by Proposition \ref{4.Prop2}, there is an indecomposable projective $A$-module $P$ that lies in $Im(F)$. Since $A$ is weakly symmetric, via a top-to-socle factorisation, we conclude that also $rad(P)$ belongs to $Im(F)$. Say $P$ is given by $P_i$ for $i\in Q_0$. Again by Proposition \ref{4.Prop2}, the projective $A$-cover of $rad(P_i)$ lies in $Im(F)$ and, thus, so do all the $P_j$ for $i\ra j$ being an arrow in $Q_1$. Now we repeat the argument with all such $P_j$. Since $A$ is connected and self-injective (every arrow in $Q_1$ is part of an oriented cycle), after finitely many steps, we conclude that $Im(F)$ contains the regular module $A$. Thus, $F$ is an equivalence.
\end{proof}

Note that, contrary to the case of group algebras, in general, derived simple self-injective algebras can admit many non-trivial homological embeddings as discussed in the following example.

\begin{example}
Let $Q$ be an oriented cycle with $n>1$ vertices and consider the self-injective Nakayama algebra $A=\Kbb Q/R^h$ with $n\mid h$ and where $R$ denotes the arrow ideal of $\Kbb Q$. It is not hard to check that for every bounded complex $C^\bullet$ of indecomposable projective $A$-modules
$$\xymatrix{...\ar[r] & 0\ar[r] & P^{-m}\ar[r] & ...\ar[r] & P^{-1}\ar[r]^d & P^0\ar[r] & 0\ar[r] & ...}$$
with $d\not= 0$ and $d$ not an isomorphism, we have $Hom_{\Dcal(A)}(C^\bullet,C^\bullet[-1])\not= 0$. A non-trivial map $C^\bullet\ra C^\bullet[-1]$ in $\Dcal(A)$ can be constructed using the isomorphism of $A$-modules $ker(d)\cong coker(d)$. 
Thus, every compact object $C^\bullet$ in $\Dcal(A)$ with $Hom_{\Dcal(A)}(C^\bullet,C^\bullet[i])=0$ for all $i\not= 0$ is of the form
$$\xymatrix{...\ar[r] & 0\ar[r] & P\ar[r] & 0\ar[r] & ...}$$
where $P$ denotes a projective $A$-module.
Using \cite[Proposition 4.1]{LY} and \cite[Lemma 2.9]{AKLY}, it follows that $A$ is derived simple.

On the other hand, there are many different equivalence classes of non-zero homological embeddings $mod(B)\ra mod(A)$ (see \cite[Theorem 5.14]{M}). For example, if $n=h$, their number is given by $2^n-1$. For a concrete example take $n=h=3$. Then $A$ is defined by the quiver
$$\xymatrix{1\ar[rr]^\alpha & & 2\ar[dl]^\beta\\ & 3\ar[ul]^\gamma &}$$
with respect to the relations $\gamma\beta\alpha=\beta\alpha\gamma=\alpha\gamma\beta=0$. Note that there are nine different isomorphism classes of indecomposable $A$-modules. The non-zero homological embeddings are, up to equivalence, determined by the following wide subcategories of $mod(A)$: 
$$add(P_1), add(P_2), add(P_3),$$ 
$$add(P_1\oplus P_2 \oplus S_1\oplus P_2/rad^2(P_2)),$$ 
$$add(P_1\oplus P_3 \oplus S_3\oplus P_1/rad^2(P_1)),$$ 
$$add(P_2\oplus P_3 \oplus S_2\oplus P_3/rad^2(P_3)),$$
$$mod(A).$$
\end{example}


\section{Homological embeddings for preprojective algebras}

We are now able to complement Theorem \ref{Main 1} of Section 4 by classifying the homological embeddings for preprojective algebras of Dynkin type. First, recall that for every algebra there are two equivalence classes of trivial homological embeddings, namely those for which the given functor is zero or an equivalence. If $A$ is a preprojective algebra of Dynkin type, these classes correspond, respectively, to the longest element and the identity element of the associated Weyl group.

\begin{example}\label{ex hom preproj}
Let $A$ be a preprojective algebra of type $\Acal_n$ $({\small{\xymatrix{1\ar@{-}[r] & 2\ar@{-}[r] & \dots\ar@{-}[r] & n}}})$ for $n\ge 2$ and let $P$ be the indecomposable projective $A$-module associated to vertex 1. In particular, we have $End_A(P)\cong\Kbb$. Consequently, $add(P)$ is a functorially finite and wide subcategory of $mod(A)$ and the associated weakly homological embedding $F:mod(M_n(\Kbb))\ra mod(A)$ is homological (see Proposition \ref{Prop ring epi} and Proposition \ref{4.Prop2}). Moreover, following Section 4, there is a torsion class in $mod(A)$ corresponding to $F$ (the smallest torsion class containing $P$) which is given by $gen(I_w)$ for $$w:=s_n(s_{n-1}s_n)\cdots (s_2s_3\cdots s_n).$$ To see this, first note that $P$ is a direct summand of the $A$-module $I_w$. Conversely, it suffices to observe that $w$ is the longest Weyl group element that can be written as a product of simple reflections excluding $s_1$. It follows that the homological embedding $F$ corresponds to the Weyl group element $w$ under the bijection in Theorem \ref{Main 1}. Analogously, when starting with the indecomposable projective $A$-module associated to vertex $n$, one obtains a further homological embedding that corresponds to the Weyl group element $s_{n-1}(s_{n-2}s_{n-1})\cdots (s_1s_2\cdots s_{n-1})$.
\end{example}

Next, we classify all homological embeddings for preprojective algebras of Dynkin type.

\begin{theorem}\label{4.Thm.2}
Let $A$ be a preprojective algebra of Dynkin type and $F:mod(B)\ra mod(A)$ be a homological embedding that is neither zero nor an equivalence. Then $A$ must be of type $\Acal_n$ $({\small{\xymatrix{1\ar@{-}[r] & 2\ar@{-}[r] & \dots\ar@{-}[r] & n}}})$ for $n\ge 2$ and the algebra $B$ is Morita equivalent to $\Kbb$\,.
In fact, for each $n\ge 2$ there are precisely two such choices for $F$, up to equivalence, which correspond to the Weyl group elements $s_{n-1}(s_{n-2}s_{n-1})\cdots (s_1s_2\cdots s_{n-1})$ and $s_n(s_{n-1}s_n)\cdots (s_2s_3\cdots s_n)$, respectively.
\end{theorem}

\begin{proof}
First of all, recall that $A$ is a periodic algebra (see \cite[IV., Theorem 14.1]{SY}) and, hence, $A$ fulfils the Tachikawa conjecture, by Lemma \ref{4.Lem2}. Now let $F:mod(B)\ra mod(A)$ be a homological embedding that is neither zero nor an equivalence. Further, suppose that $Im(F)$ contains a non-projective indecomposable $A$-module $X$. Consequently, by Proposition \ref{4.Prop2}, $Im(F)$ contains $\Omega^s(X)$ for all integers $s$ and, thus, using that $\underline{mod}(A)$ is 2-Calabi-Yau, it also contains $\nu(X)$, where $\nu$ denotes the Nakayama functor. In particular, the projective $A$-covers $P^X$ of $X$ and $P^{\nu(X)}=\nu(P^X)$ of $\nu(X)$ belong to $Im(F)$. Now let $P=P_i$ for $i\in Q_0$ be an indecomposable direct summand of $P^X$. Consider the following short exact sequence of $A$-modules
$$\xymatrix{0\ar[r] & rad(P)\ar[r] & P\ar[r] & top(P)=soc(\nu(P))\ar[r] & 0.}$$
Since $P$ and $\nu(P)$ belong to $Im(F)$, so does $rad(P)$, as the kernel of the induced map from $P$ to $\nu(P)$, and $top(P)$, as the image of this map. Note that $rad(P)$ cannot have any projective direct summand, since the algebra $A$ is self-injective. It follows that also the projective $A$-cover of $rad(P)$
$$P^{rad(P)}=\underset{\underset{\alpha:i\ra j}{\alpha\in \overline{Q}_1}}{\bigoplus}P_j$$
and, using the same arguments as above, its Nakayama shift $\nu(P^{rad(P)})$ belong to $Im(F)$. Keeping in mind the shape of the underlying quiver $\overline{Q}$, after finitely many steps, we conclude that $Im(F)$ contains the regular module $A$. Therefore, $F$ must be an equivalence, contradicting our assumption. It follows that $Im(F)$ can only contain projective $A$-modules. Since $Im(F)$ is closed for kernels and cokernels in $mod(A)$ and $Hom_A(P,Q)\not= 0$ for all indecomposable projective $A$-modules $P$ and $Q$, the essential image of $F$ must be of the form $add(P)$ for $P$ projective with $End_A(P)\cong\Kbb$. But such projectives do only exist in case $A$ is of type $\Acal_n$. In fact, for each $n\ge 2$ there are precisely two indecomposable projective $A$-modules with trivial endomorphism ring which correspond to the outer vertices of $\Acal_n$. Now we can refer to Example \ref{ex hom preproj} to complete the proof.
\end{proof}


\end{document}